\documentclass[reqno]{amsart}
\usepackage{amsmath,amsthm,amssymb,amscd}
\newtheorem{theorem}{Theorem}[section]
\newtheorem{lemma}[theorem]{Lemma}

\theoremstyle{definition}
\newtheorem{definition}[theorem]{Definition}

\theoremstyle{remark}

\numberwithin{equation}{section}
\allowdisplaybreaks
\begin{document}

%-------------------------------------------------------------------------
% editorial commands: to be inserted by the editorial office
%
%\firstpage{1} \volume{228} \Copyrightyear{2004} \DOI{003-0001}
%\seriesextraline{This is the Concrete Title of this Book\br H.E. R and S.T.C. W, Eds.}
%\firstpage{1}

\title[$\sigma_o\text{mex}(n)$ and $\sigma_e\text{mex}(n)$]
 {Arithmetic properties and asymptotic formulae for $\sigma_o\text{mex}(n)$ and $\sigma_e\text{mex}(n)$}

\author{Gurinder Singh}
\address{Department of Mathematics, Indian Institute of Technology Guwahati, Assam, India, PIN- 781039}
\email{gurinder.singh@iitg.ac.in}

\author{Rupam Barman}
\address{Department of Mathematics, Indian Institute of Technology Guwahati, Assam, India, PIN- 781039}
\email{rupam@iitg.ac.in}



\dedicatory{}
%%% ----------------------------------------------------------------------

\begin{abstract} 
The minimal excludant of an integer partition is the least positive integer missing from the partition. Let $\sigma_o\text{mex}(n)$ (resp., $\sigma_e\text{mex}(n)$) denote the sum of odd (resp., even) minimal excludants over all the partitions of $n$. Recently, Baruah et al. proved a few congruences for these partition functions modulo $4$ and $8$, and asked for asymptotic formulae for the same. In this article, we study the lacunarity of $\sigma_o\text{mex}(n)$ and $\sigma_e\text{mex}(n)$ modulo arbitrary powers of $2$ and also prove some infinite families of congruences for $\sigma_o\text{mex}(n)$ and $\sigma_e\text{mex}(n)$ modulo $4$ and $8$. We also obtain Hardy-Ramanujan type asymptotic formulae for both $\sigma_o\text{mex}(n)$ and $\sigma_e\text{mex}(n)$.
\end{abstract}
\maketitle
%%%%%%%%%%%%%%%%%%%%%%%%%%%%%%
\section{Introduction and statement of results} 
A partition of a positive integer $n$ is a non-increasing sequence of positive integers, called parts, whose sum is $n$. The number of partitions of $n$ is denoted by $p(n)$. In 2015, Fraenkel and Peled \cite{Fraenkel2015}, working in the area of game theory, defined the term minimal excludant for any set $S$ of positive integers as the least positive integer missing from $S$. Later in 2019, Andrews and Newman \cite{Andrews2019} introduced this in partition theory. They defined the minimal excludant of an integer partition $\pi$, denoted by $\text{mex}(\pi)$, as the least positive integer missing from the partition. With this they also considered the sum of minimal excludants over all the partitions of $n$, denoted by $\sigma\text{mex}(n)$:
\begin{align*}
\sigma\text{mex}(n):=\sum_{\pi\in \mathcal{P}(n)}\text{mex}(\pi),
\end{align*}
where $\mathcal{P}(n)$ is the set of all partitions of $n$. For example, the values of minimal excludants for each partition of $n=4$ are: $\text{mex}(4)=1$; $\text{mex}(3+1)=2$; $\text{mex}(2+2)=1$; $\text{mex}(2+1+1)=3$; $\text{mex}(1+1+1+1)=2$, with $\sigma\text{mex}(4)=1+2+1+3+2=9$. The works of Andrews and Newman \cite{Andrews2019, Andrews2020} motivated great research in the theory of partitions. Many mathematicians have introduced the concept of minimal excludant parts for various restricted partition functions, see, for example, \cite{Ballantine2020, Barman2021a, Barman2021b, Chakraborty, Silva2020, Hopkins2022a, Hopkins2022b, Kang2021, Kang2023, Kaur_2022}. Very recently, Baruah et al. \cite{Baruah2022} refined the arithmetic function $\sigma\text{mex}(n)$ by considering the sum of odd and even minimal excludants separately. More specifically, for a positive integer $n$, Baruah et al. \cite{Baruah2022} defined the following two arithmetic functions:
\begin{align*}
\sigma_o\text{mex}(n)&:=\sum_{\substack{\pi\in \mathcal{P}(n)\\ 2\nmid \text{mex}(\pi)}}\text{mex}(\pi),\\
\sigma_e\text{mex}(n)&:=\sum_{\substack{\pi\in \mathcal{P}(n)\\ 2\mid \text{mex}(\pi)}}\text{mex}(\pi).
\end{align*}
For example, for $n=4$, $\sigma_o\text{mex}(4)=1+1+3=5$ and $\sigma_o\text{mex}(4)=2+2=4$. Note that, for all $n\geq0$, $\sigma\text{mex}(n)=\sigma_o\text{mex}(n)+\sigma_e\text{mex}(n)$. Baruah et al. \cite{Baruah2022} established two identities involving $\sigma_o\text{mex}(n)$ and $\sigma_e\text{mex}(n)$. More precisely, they proved that
\begin{align}
G_o(q):=\sum_{n=0}^{\infty}\sigma_o\text{mex}(n)q^n&=\frac{1}{2}\left((-q;q)^2_{\infty}+(q;q)^2_{\infty}\right)=\frac{1}{2}\left(\frac{f^2_2}{f^2_1}+f^2_1\right) ,\label{5.1.4}\\
G_e(q):=\sum_{n=0}^{\infty}\sigma_e\text{mex}(n)q^n&=\frac{1}{2}\left((-q;q)^2_{\infty}-(q;q)^2_{\infty}\right)=\frac{1}{2}\left(\frac{f^2_2}{f^2_1}-f^2_1\right).\label{5.1.5}
\end{align}
Here and throughout this article, for $\lvert q\rvert<1$,
\begin{align*}
(a;q)_{\infty}:=\prod_{n=1}^{\infty}(1-aq^{n-1})
\end{align*}
and for any positive integer $m$, $f_m:=(q^m;q^m)_{\infty}$. Using \eqref{5.1.4} and \eqref{5.1.5}, following three Ramanujan-type congruences modulo $4$ and $8$ for $\sigma_o\text{mex}(n)$ and $\sigma_e\text{mex}(n)$ were also established in \cite{Baruah2022}:
\begin{align}
\sigma_o\text{mex}(2n+1)&\equiv0\pmod{4},\label{5.1.6}\\
\sigma_o\text{mex}(4n+1)&\equiv0\pmod{8},\label{5.1.7}\\
\sigma_e\text{mex}(4n)&\equiv0\pmod{4}.\label{5.1.8}
\end{align}
Very recently, Du and Tang \cite{Du_Tang2023} proved another two congruences for $\sigma_o\text{mex}(n)$ and $\sigma_e\text{mex}(n)$ modulo $8$ and $16$ which were conjectured by Baruah et al. in \cite{Baruah2022}:
\begin{align*}
\sigma_o\text{mex}(8n+1)&\equiv0\pmod{16},\\
\sigma_e\text{mex}(8n)&\equiv0\pmod{8}.
\end{align*}
\par The difficulty of proving congruences for $\sigma_o\text{mex}(n)$ and $\sigma_e\text{mex}(n)$ modulo powers of $2$ from \eqref{5.1.4} and \eqref{5.1.5} is due to the factor $\frac{1}{2}$ on the right hand sides of \eqref{5.1.4} abd \eqref{5.1.5}. In this article, we prove the following theorem which provides two new congruences for $\sigma_e\text{mex}(n)$ modulo $4$. We use Ramanujan's theta functions and certain identities involving them in the proof.
\begin{theorem}\label{thm5.1}
For all $n\geq0$, we have
\begin{align}
\sigma_e\emph{mex}(10n+6)&\equiv0\pmod{4},\label{thm1.5.1}\\
\sigma_e\emph{mex}(10n+8)&\equiv0\pmod{4}.\label{thm1.5.2}
\end{align}
\end{theorem}
We develop a relation between $\sigma_o\text{mex}(n)$ and $\sigma_e\text{mex}(n)$, and use \eqref{5.1.6}-\eqref{5.1.8} to establish interesting families of infinitely many congruences for $\sigma_o\text{mex}(n)$ and $\sigma_e\text{mex}(n)$ modulo $4$ and $8$. More precisely, we have the following theorem in which we prove Ramanujan-type congruences for $\sigma_o\text{mex}(n)$ and $\sigma_e\text{mex}(n)$ with infinitely many primes involved in each family.
\begin{theorem}\label{thm5.2}
Let $n\geq0$.
\begin{enumerate}
\item For all primes $p\equiv5,7,11\pmod{12}$ and odd integers $1\leq k<p$, we have
\begin{align}\label{5.1.13}
\sigma_e\emph{mex}\left(2p^2n+kp+\frac{p^2-1}{12}\right)\equiv0\pmod{4}.
\end{align}
\item For all primes $p\equiv5,7,11\pmod{24}$, we have
\begin{align}\label{5.1.14}
\sigma_e\emph{mex}\left(4p^2n+kp+\frac{p^2-1}{12}\right)\equiv0\pmod{8},
\end{align}
where $1\leq k<p$ with $k\equiv
\begin{cases}
1\pmod{4}, & \emph{if}\ p\equiv11\pmod{24};\\
3\pmod{4}, & \emph{if}\ p\equiv5,7\pmod{24}.
\end{cases}
$
\item For all primes $p\equiv5,7,11\pmod{24}$, we have
\begin{align}\label{5.1.15}
\sigma_o\emph{mex}\left(4p^2n+kp+\frac{p^2-1}{12}\right)\equiv0\pmod{4},
\end{align}
where $1\leq k<p$ with $k\equiv
\begin{cases}
0\pmod{4}, & \emph{if}\ p\equiv5,11\pmod{24};\\
2\pmod{4}, & \emph{if}\ p\equiv7\pmod{24}.
\end{cases}
$
\end{enumerate}
\end{theorem}
\par In addition to the study of Ramanujan-type congruences, we also study the distribution of $\sigma_o\text{mex}(n)$ and $\sigma_e\text{mex}(n)$ modulo arbitrary powers of $2$. Given an integral power series $F(q):=\sum_{n=0}^{\infty}a(n)q^n$ and $0\leq r<M$, we define
\begin{align*}
\delta_r(F, M; X):=\frac{\#\{n\leq X: a(n)\equiv r \pmod{M}\}}{X}.
\end{align*}
An integral power series $F$ is called \textit{lacunary modulo $M$} if 
\begin{align*}
\lim _{X\rightarrow \infty}\delta_0(F, M; X)=1,
\end{align*}
that is, ``almost all" of the coefficients of $F$ are divisible by $M$. 
\par In this article, we study lacunarity of $G_o$ and $G_e$ modulo arbitrary powers of $2$. More precisely, we prove that for any positive integer $k$, $\sigma_o\text{mex}(n)$ and $\sigma_e\text{mex}(n)$ are divisible by $2^k$ for almost all $n$.
\begin{theorem} \label{thm5.3}
For any positive integer $k$, the series $G_o(q)=\sum_{n=0}^{\infty}\sigma_o\emph{mex}(n)q^n$ is lacunary modulo $2^k$, that is,
\begin{align*}
\lim _{X\rightarrow \infty}\delta_0(G_o, 2^k; X)=1.
\end{align*}
\end{theorem}  
\begin{theorem} \label{thm5.4}
For any positive integer $k$, the series $G_e(q)=\sum_{n=0}^{\infty}\sigma_e\emph{mex}(n)q^n$ is lacunary modulo $2^k$, that is, 
\begin{align*}
\lim _{X\rightarrow \infty}\delta_0(G_e, 2^k; X)=1.
\end{align*}
\end{theorem} 
\par 
The study of asymptotic behavior of partition functions has also been an integral part of research in the theory of partitions. Using their most celebrated circle method, Hardy and Ramanujan \cite{Hardy} established an asymptotic formula for the partition function $p(n)$, namely
\begin{align*}
p(n)\sim\frac{1}{4\sqrt{3}n}\exp\left( \pi\sqrt{\frac{2n}{3}}\right)\quad \text{as}\ n\rightarrow\infty.
\end{align*}
Grabner and Knopfmacher \cite{Grabner2006} obtained the Hardy-Ramanujan type asymptotic formula for $\sigma\text{mex}(n)$, though they were working on ``smallest gap" in a partition that has exactly the same meaning as that of minimal excludant. The asymptotic formula for $\sigma\text{mex}(n)$ due to Grabner and Knopfmacher \cite{Grabner2006} is as follows: 
\begin{align}\label{Grab_Knop}
\sigma\text{mex}(n)\sim\frac{1}{4\sqrt[4]{6n^3}}\exp\left( \pi\sqrt{\frac{2n}{3}}\right)\quad \text{as}\ n\rightarrow\infty.
\end{align}
In \cite{Baruah2022}, Baruah et al. asked for the Hardy-Ramanujan type asymptotic formulae for $\sigma_o\text{mex}(n)$ and $\sigma_e\text{mex}(n)$. In this article, we prove that as $n\rightarrow\infty$, $\sigma_o\text{mex}(n)$ and $\sigma_e\text{mex}(n)$ behave same. We use Ingham's Tauberian theorem to derive asymptotic formulae for $\sigma_o\text{mex}(n)$ and $\sigma_e\text{mex}(n)$. More specifically, we have the following result. 
\begin{theorem}\label{thm5.5}
We have
\begin{align*}
\sigma_o\emph{mex}(n)\sim\sigma_e\emph{mex}(n)\sim\frac{1}{8\sqrt[4]{6n^3}}\exp\left( \pi\sqrt{\frac{2n}{3}}\right)
\end{align*}
as $n\rightarrow\infty$.
\end{theorem}
We see that Theorem \ref{thm5.5} immediately implies \eqref{Grab_Knop} since
$$
\sigma\text{mex}(n)=\sigma_o\text{mex}(n)+\sigma_e\text{mex}(n)\sim\frac{1}{4\sqrt[4]{6n^3}}\exp\left( \pi\sqrt{\frac{2n}{3}}\right)
$$
as $n\rightarrow\infty$.
%%%%%%%%%%%%%%%%%%%%%%%%
\section{Preliminaries}
We recall some definitions and basic facts on modular forms. For more details, see for example \cite{koblitz1993, ono2004}. We first define the matrix groups 
\begin{align*}
\text{SL}_2(\mathbb{Z}) & :=\left\{\begin{bmatrix}
a  &  b \\
c  &  d      
\end{bmatrix}: a, b, c, d \in \mathbb{Z}, ad-bc=1
\right\},\\
\Gamma_{0}(N) & :=\left\{
\begin{bmatrix}
a  &  b \\
c  &  d      
\end{bmatrix} \in \text{SL}_2(\mathbb{Z}) : c\equiv 0\pmod N \right\},
\end{align*}
\begin{align*}
\Gamma_{1}(N) & :=\left\{
\begin{bmatrix}
a  &  b \\
c  &  d      
\end{bmatrix} \in \Gamma_0(N) : a\equiv d\equiv 1\pmod N \right\},
\end{align*}
and 
\begin{align*}\Gamma(N) & :=\left\{
\begin{bmatrix}
a  &  b \\
c  &  d      
\end{bmatrix} \in \text{SL}_2(\mathbb{Z}) : a\equiv d\equiv 1\pmod N, ~\text{and}~ b\equiv c\equiv 0\pmod N\right\},
\end{align*}
where $N$ is a positive integer. A subgroup $\Gamma$ of $\text{SL}_2(\mathbb{Z})$ is called a congruence subgroup if $\Gamma(N)\subseteq \Gamma$ for some $N$. The smallest $N$ such that $\Gamma(N)\subseteq \Gamma$
is called the level of $\Gamma$. For example, $\Gamma_0(N)$ and $\Gamma_1(N)$
are congruence subgroups of level $N$. 
\par Let $\mathbb{H}:=\{z\in \mathbb{C}: \text{Im}(z)>0\}$ be the upper half of the complex plane. The group $$\text{GL}_2^{+}(\mathbb{R})=\left\{\begin{bmatrix}
a  &  b \\
c  &  d      
\end{bmatrix}: a, b, c, d\in \mathbb{R}~\text{and}~ad-bc>0\right\}$$ acts on $\mathbb{H}$ by $\begin{bmatrix}
a  &  b \\
c  &  d      
\end{bmatrix} z=\displaystyle \frac{az+b}{cz+d}$.  
We identify $\infty$ with $\displaystyle\frac{1}{0}$ and define $\begin{bmatrix}
a  &  b \\
c  &  d      
\end{bmatrix} \displaystyle\frac{r}{s}=\displaystyle \frac{ar+bs}{cr+ds}$, where $\displaystyle\frac{r}{s}\in \mathbb{Q}\cup\{\infty\}$.
This gives an action of $\text{GL}_2^{+}(\mathbb{R})$ on the extended upper half-plane $\mathbb{H}^{\ast}=\mathbb{H}\cup\mathbb{Q}\cup\{\infty\}$. 
Suppose that $\Gamma$ is a congruence subgroup of $\text{SL}_2(\mathbb{Z})$. A cusp of $\Gamma$ is an equivalence class in $\mathbb{P}^1=\mathbb{Q}\cup\{\infty\}$ under the action of $\Gamma$.
\par The group $\text{GL}_2^{+}(\mathbb{R})$ also acts on functions $f: \mathbb{H}\rightarrow \mathbb{C}$. In particular, suppose that $\gamma=\begin{bmatrix}
a  &  b \\
c  &  d      
\end{bmatrix}\in \text{GL}_2^{+}(\mathbb{R})$. If $f(z)$ is a meromorphic function on $\mathbb{H}$ and $\ell$ is an integer, then define the slash operator $|_{\ell}$ by 
$$(f|_{\ell}\gamma)(z):=(\text{det}~{\gamma})^{\ell/2}(cz+d)^{-\ell}f(\gamma z).$$
\begin{definition}
	Let $\Gamma$ be a congruence subgroup of level $N$. A holomorphic function $f: \mathbb{H}\rightarrow \mathbb{C}$ is called a modular form with integer weight $\ell$ on $\Gamma$ if the following hold:
	\begin{enumerate}
		\item We have $$f\left(\displaystyle \frac{az+b}{cz+d}\right)=(cz+d)^{\ell}f(z)$$ for all $z\in \mathbb{H}$ and all $\begin{bmatrix}
		a  &  b \\
		c  &  d      
		\end{bmatrix} \in \Gamma$.
		\item If $\gamma\in \text{SL}_2(\mathbb{Z})$, then $(f|_{\ell}\gamma)(z)$ has a Fourier expansion of the form $$(f|_{\ell}\gamma)(z)=\displaystyle\sum_{n\geq 0}a_{\gamma}(n)q_N^n,$$
		where $q_N:=e^{2\pi iz/N}$.
	\end{enumerate}
\end{definition}
For a positive integer $\ell$, the complex vector space of modular forms of weight $\ell$ with respect to a congruence subgroup $\Gamma$ is denoted by $M_{\ell}(\Gamma)$.
\begin{definition}\cite[Definition 1.15]{ono2004}
	If $\chi$ is a Dirichlet character modulo $N$, then we say that a modular form $f\in M_{\ell}(\Gamma_1(N))$ has Nebentypus character $\chi$ if
	$$f\left( \frac{az+b}{cz+d}\right)=\chi(d)(cz+d)^{\ell}f(z)$$ for all $z\in \mathbb{H}$ and all $\begin{bmatrix}
	a  &  b \\
	c  &  d      
	\end{bmatrix} \in \Gamma_0(N)$. The space of such modular forms (resp. cusp forms) is denoted by $M_{\ell}(\Gamma_0(N), \chi)$. 
\end{definition}
In this paper, the relevant modular forms are those that arise from eta-quotients. Recall that the Dedekind eta-function $\eta(z)$ is defined by
\begin{align*}
\eta(z):=q^{1/24}\prod_{n=1}^{\infty}(1-q^n),
\end{align*}
where $q:=e^{2\pi iz}$ and $z\in \mathbb{H}$. It satisfies the following modular transformation property \cite[Theorem 1.61]{ono2004}: For $z\in\mathbb{H}$,
\begin{align}\label{eta_fn_transf}
\eta(-1/z)=(-iz)^{1/2}\eta(z).
\end{align}
A function $f(z)$ is called an eta-quotient if it is of the form
\begin{align*}
f(z)=\prod_{\delta\mid N}\eta(\delta z)^{r_\delta},
\end{align*}
where $N$ is a positive integer and each $r_{\delta}$ is an integer. We now recall two theorems from \cite[p. 18]{ono2004} which are very useful in checking modularity of eta-quotients.
\begin{theorem}\cite[Theorem 1.64]{ono2004}\label{thm_ono1} If $f(z)=\prod_{\delta\mid N}\eta(\delta z)^{r_\delta}$ 
	is an eta-quotient such that $\ell=\frac{1}{2}\sum_{\delta\mid N}r_{\delta}\in \mathbb{Z}$, 
	$$\sum_{\delta\mid N} \delta r_{\delta}\equiv 0 \pmod{24}$$ and
	$$\sum_{\delta\mid N} \frac{N}{\delta}r_{\delta}\equiv 0 \pmod{24},$$
	then $f(z)$ satisfies $$f\left( \frac{az+b}{cz+d}\right)=\chi(d)(cz+d)^{\ell}f(z)$$
	for every  $\begin{bmatrix}
	a  &  b \\
	c  &  d      
	\end{bmatrix} \in \Gamma_0(N)$. Here the character $\chi$ is defined by $\chi(d):=\left(\frac{(-1)^{\ell} s}{d}\right)$, where $s:= \prod_{\delta\mid N}\delta^{r_{\delta}}$. 
\end{theorem}
Suppose that $f$ is an eta-quotient satisfying the conditions of Theorem \ref{thm_ono1} and that the associated weight $\ell$ is a positive integer. If $f(z)$ is holomorphic at all of the cusps of $\Gamma_0(N)$, then $f(z)\in M_{\ell}(\Gamma_0(N), \chi)$. The following theorem due to Ligozat gives the necessary criterion for determining orders of an eta-quotient at cusps.
\begin{theorem}\cite[Theorem 1.65]{ono2004}\label{thm_ono2}
	Let $c, d$ and $N$ be positive integers with $d\mid N$ and $\gcd(c, d)=1$. If $f$ is an eta-quotient satisfying the conditions of Theorem \ref{thm_ono1} for $N$, then the 
	order of vanishing of $f(z)$ at the cusp $\frac{c}{d}$ 
	is $$\frac{N}{24}\sum_{\delta\mid N}\frac{\gcd(d,\delta)^2r_{\delta}}{\gcd(d,\frac{N}{d})d\delta}.$$
\end{theorem}
Next, we recall a theorem of Serre on the divisibility of Fourier coefficients of any modular form $f(z)\in M_{\ell}(\Gamma_0(N), \chi)$.
\begin{theorem}\cite[Theorem 2.65]{ono2004}\label{Serre_thm}
If $f(z)\in M_{\ell}(\Gamma_0(N), \chi)$ has Fourier expansion 
$$f(z)=\sum_{n=0}^{\infty}c(n)q^n\in \mathbb{Z}[[q]],$$
then, for given any positive integer $m$, there is a constant $\alpha>0$  such that	
$$ \# \left\{n\leq X: c(n)\not\equiv 0 \pmod{m} \right\}= \mathcal{O}\left(\frac{X}{(\log{}X)^{\alpha}}\right).$$
\end{theorem}
Notice that, Theorem \ref{Serre_thm} implies that for such a modular form $f(z)$ and for every positive integer $m$, ``almost all" of the $c(n)$ are zero modulo $m$. More precisely, 
\begin{align*}
\lim _{X\rightarrow \infty}\delta_0(f, m; X)=\lim_{X\to\infty} \frac{\# \left\{n\leq X: c(n)\equiv 0 \pmod{m}\right\}}{X}&=1.
\end{align*}
\par The following lemma is an easy consequence of binomial theorem.
\begin{lemma}\label{lemma_binom}
For positive integers $m$ and $k$, we have
\begin{align*}
f^{2^k}_m\equiv f^{2^{k-1}}_{2m}\pmod{2^k}.
\end{align*}
\end{lemma}
We will frequently use the congruences in Lemma \ref{lemma_binom} without explicitly mentioning them. The next lemma contains some known $5$-dissections of certain $q$-products.
\begin{lemma}\label{lemma2.6}
We have
\begin{align}
f_1&=f_{25}\left(R(q^5)^{-1}-q-q^2R(q^5) \right),\label{2.1}\\
\frac{1}{f_1}&=\frac{f^5_{25}}{f^6_5}\left( R(q^5)^{-4}+qR(q^5)^{-3}+2q^2R(q^5)^{-2}+3q^3R(q^5)^{-1}+5q^4-3q^5R(q^5)\right.\nonumber\\
&\quad\left.+2q^6R(q^5)^2-q^7R(q^5)^3+q^8R(q^5)^4\right),\label{2.2}
\end{align}
where $\displaystyle R(q)=\frac{(q;q^5)_{\infty}(q^4;q^5)_{\infty}}{(q^2;q^5)_{\infty}(q^3;q^5)_{\infty}}$.
\end{lemma}
\begin{proof}
Equation \eqref{2.1} is \cite[eq. (8.1.1)]{hirschhorn} and equation \eqref{2.2} is \cite[eq. (8.4.4)]{hirschhorn}.
\end{proof}
%%%%%%%%%%%%%%%%%%%%%%%%%%%%%%%%%%%%%%_Section_3_%%%%%%%%%%%%%%%%%%%%%%%%%%%%%%%%%%%%
\section{Proof of Theorem \ref{thm5.1}}
To prove Theorem \ref{thm5.1}, we first recall Ramanujan's general theta function $f(a,b)$ and two special cases $\varphi(q)$ and $\psi(q)$ of that along with certain identities involving them. For more details, see, for example \cite{berndtIII,berndt}. Ramanujan's general theta function
$f(a,b)$, which is defined for $|ab|<1$, is given by
\begin{align*}
f(a,b):=\sum_{n=-\infty}^{\infty}a^{\frac{n(n+1)}{2}}b^{\frac{n(n-1)}{2}},
\end{align*}
and its special cases are:
\begin{align}
\varphi(q)&:=f(q,q)=\sum_{n=-\infty}^{\infty}q^{n^2}=\frac{f^5_2}{f^2_1f^2_4},\label{5.3.1}\\
\psi(q)&:=f(q,q^3)=\sum_{n=0}^{\infty}q^{{n+1 \choose 2}}=\frac{f^2_2}{f_1}, \label{5.3.2}\\
\varphi(-q)&=f(-q,-q)=\sum_{n=-\infty}^{\infty}(-1)^nq^{n^2}=\frac{f^2_1}{f_2}.\label{5.3.3}
\end{align} 
Last equalities in \eqref{5.3.1}, \eqref{5.3.2} and \eqref{5.3.3} are due to Jacobi's triple product identity given by \cite[p. 10]{berndt}
\begin{align*}
f(a,b)=(-a;ab)_{\infty}(-b;ab)_{\infty}(ab;ab)_{\infty}.
\end{align*}
Now, from \eqref{5.3.1} and \eqref{5.3.3}, we have
\begin{align}\label{5.3.4}
\varphi(-q)=\frac{\varphi^2(-q^2)}{\varphi(q)}.
\end{align}
We next recall two identities from \cite[p. 40]{berndtIII} which relate $\varphi(q)$ and $\psi(q)$, namely
\begin{align}
\varphi(q)+\varphi(-q)=2\varphi(q^4),\label{5.3.5}\\
\varphi(q)-\varphi(-q)=4q\psi(q^8).\label{5.3.6}
\end{align}
From \eqref{5.3.5} and \eqref{5.3.6}, we deduce that
\begin{align}
\varphi(q)=\varphi(q^4)+2q\psi(q^8),\label{5.3.7}\\
\varphi(-q)=\varphi(q^4)-2q\psi(q^8).\label{5.3.8}
\end{align}
The Lambert series representation of $\varphi^2(q)$ is given by \cite[p. 58]{berndt}
\begin{align}\label{5.3.9}
\varphi^2(q)=1+4\sum_{n=1}^{\infty}\frac{q^n}{1+q^{2n}}.
\end{align}
Lastly, we recall from \cite[p. 107]{berndt}
\begin{align}\label{5.3.10}
\sum_{n=1}^{\infty}\left(\frac{n}{5}\right)\frac{q^n}{(1-q^n)^2}=q\frac{(q^5;q^5)^5_{\infty}}{(q;q)_{\infty}},
\end{align}
where $\left( \frac{n}{5}\right) $ denotes the Legendre symbol. From \eqref{5.3.10}, we arrive at
\begin{align}\label{5.3.11}
\sum_{\substack{n=1\\ 5\nmid n}}^{\infty}\frac{q^n}{1+q^{2n}}\equiv q\frac{f^5_5}{f_1}\pmod{2}.
\end{align}
We are now equipped enough to prove Theorem \ref{thm5.1}.
\begin{proof}[Proof of Theorem \ref{thm5.1}] From \eqref{5.1.5}, we have
\begin{align}
\sum_{n=0}^{\infty}\sigma_e\text{mex}(n)q^n&=\frac{1}{2}\left(\frac{f^2_2}{f^2_1}-f^2_1\right)\nonumber\\
&=\frac{1}{2}f_2\left(\frac{f_2}{f^2_1}-\frac{f^2_1}{f_2}\right)\nonumber\\
&=\frac{1}{2}f_2\left(\frac{1}{\varphi(-q)}-\varphi(-q)\right)\nonumber\\
&=\frac{1}{2}f_2\left(\frac{\varphi(q)}{\varphi^2(-q^2)}-\varphi(-q)\right),\label{5.3.12}
\end{align}	
where the last equality is due to \eqref{5.3.4}. Next, invoking \eqref{5.3.7} and \eqref{5.3.8} in \eqref{5.3.12}, we deduce that
\begin{align}
\sum_{n=0}^{\infty}\sigma_e\text{mex}(n)q^n=\frac{1}{2}f_2\left(\frac{\varphi(q^4)}{\varphi^2(-q^2)}-\varphi(q^4)+2q\left( \frac{\psi(q^8)}{\varphi^2(-q^2)}+\psi(q^8)\right) \right).\label{5.3.13}
\end{align}	
Extracting terms with even powers of $q$ on both sides of \eqref{5.3.13} and then replacing $q^2$ by $q$, we obtain
\begin{align}
\sum_{n=0}^{\infty}\sigma_e\text{mex}(2n)q^n&=\frac{1}{2}f_1\left(\frac{\varphi(q^2)}{\varphi^2(-q)}-\varphi(q^2) \right),\nonumber\\
&=\frac{f_1\varphi(q^2)}{\varphi^2(-q)}\left(\frac{1-\varphi^2(-q)}{2}\right),\nonumber\\
&=\frac{f_1\varphi(q^2)}{\varphi^2(-q)}\left(2\sum_{n=1}^{\infty}\frac{(-1)^{n+1}q^n}{1+q^{2n}}\right),\label{5.3.14}
\end{align} 
where the last equality follows from \eqref{5.3.9}. Reducing \eqref{5.3.14} modulo $4$ yields
\begin{align}
\sum_{n=0}^{\infty}\sigma_e\text{mex}(2n)q^n&\equiv2f_1\sum_{n=1}^{\infty}\frac{q^n}{1+q^{2n}}\nonumber\\
&\equiv 2f_1\left(\sum_{\substack{n=1\\ 5\mid n}}^{\infty}\frac{q^n}{1+q^{2n}}+\sum_{\substack{n=1\\ 5\nmid n}}^{\infty}\frac{q^n}{1+q^{2n}}\right)\nonumber\\
&\equiv 2f_1\left(\sum_{n=1}^{\infty}\frac{q^{5n}}{1+q^{10n}}+q\frac{f^5_5}{f_1}\right)\pmod{4}, \label{5.3.15}
\end{align}
where the last equality is due to \eqref{5.3.11}. Finally, employing $5$-dissections of $f_1$ and $\frac{1}{f_1}$ from Lemma \ref{lemma2.6} in \eqref{5.3.15} and extracting terms of the type $q^{5n+3}$ on both sides, we arrive at
\begin{align}
\sum_{n=0}^{\infty}\sigma_e\text{mex}(10n+6)q^{5n+3}\equiv4q^3\frac{f^6_{25}}{f_5}\left(\frac{1}{R^2(q^5)}-\frac{1}{R^3(q^5)} \right) \pmod{4}.\label{5.3.16}
\end{align}
Clearly, \eqref{5.3.16} implies \eqref{thm1.5.1}. Similarly, employing $5$-dissections of $f_1$ and $\frac{1}{f_1}$ from Lemma \ref{lemma2.6} in \eqref{5.3.15} and extracting terms of the type $q^{5n+4}$ on both sides, we arrive at
\begin{align}
\sum_{n=0}^{\infty}\sigma_e\text{mex}(10n+8)q^{5n+4}\equiv4q^4\frac{f^5_{25}}{f_5}\left(\frac{1}{R^2(q^5)}-\frac{1}{R(q^5)} \right) \pmod{4}.\label{5.3.17}
\end{align}
Clearly, \eqref{5.3.17} implies \eqref{thm1.5.2}. This completes the proof.
\end{proof}
%%%%%%%%%%%%%%%%%%%%%%%%%%%%%
\section{Proof of Theorem \ref{thm5.2}}
In order to prove Theorem \ref{thm5.2}, we first observe the pattern in the coefficients of $q$-product $f^2_1$. Define
\begin{align*}
f^2_1:=\sum_{n=0}^{\infty}a(n)q^n.
\end{align*}
In \cite{Cooper}, Cooper, Hirschhorn, and Lewis studied the powers of Euler's product $(q;q)_{\infty}=f_1$. In particular, we use their result regarding the coefficients of $f^2_1$, which we write as a lemma below.
\begin{lemma}\cite[Theorem 1]{Cooper}\label{lemma5.4.1}
Let $p\equiv5,7,11\pmod{12}$ be a prime. Then
\begin{align}\label{5.4.1}
a\left( pn+\frac{p^2-1}{12}\right)=\epsilon\ a\left( \frac{n}{p}\right),
\end{align}
where $\epsilon=\begin{cases}
1 & p\equiv7,11\pmod{12};\\
-1 & p\equiv5\pmod{12}.
\end{cases}$
\end{lemma}
In \eqref{5.4.1}, $a\left(\frac{n}{p}\right)$ is taken to be zero whenever $\frac{n}{p}$ is not an integer. Notice an easy consequence of Lemma \ref{lemma5.4.1}: For any prime $p\equiv5,7,11\pmod{12}$, we have
\begin{align}\label{5.4.2}
a\left(p^2n+pk+\frac{p^2-1}{12}\right)=0,
\end{align}
where $k$ is an integer with $1\leq k<p$.
\begin{proof}[Proof of Theorem \ref{thm5.2}]
In view of \eqref{5.1.4} and \eqref{5.1.5}, we find that
\begin{align}
\sum_{n=0}^{\infty}\sigma_o\text{mex}(n)q^n&=\frac{1}{2}\left(\frac{f^2_2}{f^2_1}+f^2_1\right)\nonumber\\
&=\frac{1}{2}\left(\frac{f^2_2}{f^2_1}-f^2_1\right)+f^2_1\nonumber\\
&=\sum_{n=0}^{\infty}\sigma_e\text{mex}(n)q^n+\sum_{n=0}^{\infty}a(n)q^n.\label{5.4.3}
\end{align} 
Now, from \eqref{5.4.2} and \eqref{5.4.3}, we obtain that for any prime $p\equiv5,7,11\pmod{12}$, 
\begin{align}\label{5.4.4}
\sigma_o\text{mex}\left(p^2n+pk+\frac{p^2-1}{12}\right) = \sigma_e\text{mex}\left(p^2n+pk+\frac{p^2-1}{12}\right),
\end{align}
where $k$ is an integer with $1\leq k<p$.
\begin{enumerate}
\item From \eqref{5.1.6}, we have that $\sigma_o\text{mex}(n)$ is divisible by $4$ for every odd integer $n$. Also, for any prime $p\equiv5,7,11\pmod{12}$ and an odd integer $k$, $2p^2n+pk+\frac{p^2-1}{12}$ is odd. Therefore, replacing $n$ by $2n$ in \eqref{5.4.4} and reducing it modulo $4$ yields \eqref{5.1.13}.
	
\item Similarly, from \eqref{5.1.7}, we have that $\sigma_o\text{mex}(n)$ is divisible by $8$ whenever $n\equiv1\pmod{4}$. We observe that for any prime $p\equiv5,7,11\pmod{24}$, $$4p^2n+pk+\frac{p^2-1}{12}\equiv1\pmod{4},$$
for $1\leq k<p$ with $k\equiv
\begin{cases}
	1\pmod{4}, & \text{if}\  p\equiv11\pmod{24};\\
	3\pmod{4}, & \text{if}\  p\equiv5,7\pmod{24}.
\end{cases}
$ \\
Therefore, replacing $n$ by $4n$ in \eqref{5.4.4} and reducing it modulo $8$ yields \eqref{5.1.14}.

\item Next, we notice from \eqref{5.1.8} that $\sigma_e\text{mex}(n)$ is divisible by $4$ whenever $n\equiv0\pmod{4}$. Also, for any prime $p\equiv5,7,11\pmod{24}$, $$4p^2n+pk+\frac{p^2-1}{12}\equiv0\pmod{4},$$
for $1\leq k<p$ with $k\equiv
\begin{cases}
	0\pmod{4}, & \text{if}\  p\equiv5,11\pmod{24};\\
	2\pmod{4}, & \text{if}\  p\equiv7\pmod{24}.
\end{cases}
$ \\
Therefore, replacing $n$ by $4n$ in \eqref{5.4.4} and reducing it modulo $4$ yields \eqref{5.1.15}.
\end{enumerate}
This completes the proof.
\end{proof}
%%%%%%%%%%%%%%%%%%%%%%%%
\section{Proof of Theorems \ref{thm5.3} and \ref{thm5.4}}
\begin{proof}[Proof of Theorems \ref{thm5.3} and \ref{thm5.4}]
Let 
\begin{align}\label{5.5.1}
A(z):=\sum_{n=0}^{\infty}\alpha(n)q^n:=\frac{f^2_2}{f^2_1}=q^{-1/12}\frac{\eta^2(2z)}{\eta^2(z)}.
\end{align}
From \eqref{5.5.1}, we have
\begin{align*}
A^*(z):=\sum_{n=0}^{\infty}\alpha(n)q^{12n+1}=\frac{\eta^2(24z)}{\eta^2(12z)}.
\end{align*}
Here, $A^*$ is not a modular form since it is not holomorphic at cusps. We multiply $A^*$ by a suitable eta-quotient to obtain an integer weight modular form. It is easy to see using binomial theorem that for any nonnegative integer $k$,
\begin{align*}
\frac{\eta^{2^{k+1}}(12z)}{\eta^{2^k}(24z)}\equiv1\pmod{2^{k+1}}.
\end{align*}
If we define
\begin{align*}
A_k(z):=\frac{\eta^2(24z)}{\eta^2(12z)}\cdot\frac{\eta^{2^{k+1}}(12z)}{\eta^{2^k}(24z)}=\frac{\eta^{2^{k+1}-2}(12z)}{\eta^{2^k-2}(24z)},
\end{align*}
then
\begin{align}\label{5.5.3}
A_k(z)\equiv A^*(z)\pmod{2^{k+1}}.
\end{align}
Now, $A_k(z)$ is an eta-quotient with $N=288$. The cusps of $\Gamma_{0}(288)$ are represented by fractions $\frac{c}{d}$ where $d\mid 288$ and 
$\gcd(c, d)=1$. For example, see \cite{Cotron_2020}. By Theorem \ref{thm_ono2}, we find that $A_k(z)$ is holomorphic at a cusp $\frac{c}{d}$ if and only if	
\begin{align*}
12\left(\frac{\gcd(d,12)^2}{\gcd(d,288/d)^2} \frac{\left(2^{k+1}-2\right)}{12d}- \frac{\gcd(d,2)^2}{\gcd(d,288/d)^2}\frac{\left(2^{k}-2 \right)}{24d}\right)\geq 0.
\end{align*}
Equivalently, if and only if
\begin{align*}
	S:=\gcd(d,12)^2 \left(2^{k+1}-2\right) + \gcd(d,24)^2\left(1-2^{k-1}\right)\geq 0.
\end{align*}
In the following table, we list all the possible values of $S$.
\begin{center}
\begin{tabular}{|p{3.3cm}|p{3cm}|}
\hline
$d\mid 288$ &  $S$\\	
\hline	
$1$   & $2^{k-1}\cdot3-1$\\
\hline
$2$& $4\left(2^{k-1}\cdot3-1\right) $\\
\hline 
$4$& $16\left(2^{k-1}\cdot3-1\right)$\\
\hline
$8,16,32$ & $32$\\
\hline
$3,9$    & $9\left(2^{k-1}\cdot3-1\right)$\\
\hline
$6,18$& $36\left(2^{k-1}\cdot3-1\right)$\\
\hline
$12,36$& $144\left(2^{k-1}\cdot3-1\right)$\\
\hline
$24,48,72,96,144,288$& $288$\\
\hline
\end{tabular}
\end{center}	
\vspace{.5cm}	
It is clear from the table that $S\geq 0$ for all $d\mid 288$ and $k\geq0$. Therefore, by Theorem \ref{thm_ono2}, $A_k(z)$ is holomorphic at every cusp $\frac{c}{d}$. The weight of $A_k(z)$ is $\ell=2^{k}$ and the associated character is the trivial character $\chi_0$. Finally, Theorems \ref{thm_ono1} and \ref{thm_ono2} yields that $A_k(z) \in M_{2^{k}}\left(\Gamma_{0}(288), \chi_0\right)$ for $k\geq 0$. Hence, by Theorem \ref{Serre_thm} of Serre, the Fourier coefficients of $A_k(z)$ are almost always divisible by $m=2^{k+1}$. Now, using \eqref{5.5.3}, we conclude that $A^*(z)$ is lacunary modulo $2^{k+1}$, for all $k\geq0$.
\par Next, let 
\begin{align}\label{5.5.4}
B(z):=\sum_{n=0}^{\infty}\beta(n)q^n:=f^2_1=q^{-1/12}\eta^2(z).
\end{align}
From \eqref{5.5.4}, we have
\begin{align*}
B^*(z):=\sum_{n=0}^{\infty}\beta(n)q^{12n+1}=\eta^2(12z).
\end{align*}
By Theorems \ref{thm_ono1} and \ref{thm_ono2}, we find that $B^*(z)\in M_{1}\left(\Gamma_{0}(144), \chi_1\right)$, where character $\chi_1$ is given by $\chi_1(\bullet)= \left( \frac{-1}{\bullet}\right)$. Hence, by Theorem \ref{Serre_thm} of Serre, we conclude that $B^*(z)$ is lacunary modulo $m=2^{k+1}$, for all $k\geq0$. 
\par From the discussion above, we find that $A(z)$ and $B(z)$ are also lacunary modulo $2^k$, for all positive integers $k$. And, since 
$$2G_o(q)=A(z)+B(z)$$
and
$$2G_e(q)=A(z)-B(z),$$
we conclude that $G_o(q)$ and $G_e(q)$ are lacunary modulo $2^k$, for any positive integer $k$.
\end{proof}
%%%%%%%%%%%%%%%%%%%%%%%%%
\section{Proof of Theorem \ref{thm5.5}}
To obtain asymptotic formulae for $\sigma_o\text{mex}(n)$ and $\sigma_e\text{mex}(n)$, we use the Tauberian theorem of Ingham \cite{Ingham}. Ingham's theorem enables us to derive asymptotic formula for the coefficients $c(n)$ of certain power series $C(q)=\sum_{n=0}^{\infty}c(n)q^n$ from the behavior of its generating function $C(e^{-y})$ while $y\rightarrow0^+$. The following result is a special case of Ingham's Theorem, see, for example, \cite{Kaur_2022}.
\begin{theorem}\cite[Theorem 1]{Ingham}\label{theorem_Ingham}
Let $C(q)=\sum_{n=0}^{\infty}c(n)q^n$ be a power series with radius of convergence $1$. Assume that $\left\lbrace c(n)\right\rbrace$ is a weakly increasing sequence of nonnegative real numbers. If there are constants $\mu,\nu\in \mathbb{R}$, and $\lambda>0$ such that
\begin{align*}
C(e^{-y})\sim \mu y^{\nu}e^{\frac{\lambda}{y}},\quad \text{as}\ y\rightarrow0^+
\end{align*}
then we have
\begin{align*}
c(n)\sim\frac{\mu}{2\sqrt{\pi}}\frac{\lambda^{\frac{2\nu+1}{4}}}{n^{\frac{2\nu+3}{4}}}e^{2\sqrt\lambda n},\quad \text{as}\ n\rightarrow\infty.
\end{align*}
\end{theorem}
The next lemma is an easy representation of Theorem $1.1$ of Berndt and Kim \cite{Berndt_kim} in our setup. For more details on the terminology involved, see \cite{Berndt_kim}.
\begin{lemma}\cite[Theorem 1.1]{Berndt_kim}\label{lemma_Berndt_Kim}
Let $E_{2n}$, $n\geq0$, denote the $2n$-th Euler number, and let $H_n(x)$, $n\geq0$, be the $n$-th Hermite polynomial. Then, as $y\rightarrow0^+$,
\begin{align}
\sum_{n=0}^{\infty}(-1)^{n}e^{-(an^2+bn)y}\sim e^{\left(\frac{a-2b}{4}\right)y}\sum_{n=0}^{\infty}\frac{E_{2n}a^n}{(2n)!2^{2n+1}}y^n\cdot H_{2n}\left(\frac{b-a}{2\sqrt{a}}\sqrt{y}\right).\label{5.6.0}
\end{align}
\begin{proof}
Replacing $\frac{1-t}{1+t}(=q)$ by $e^{-ay}$ and then $b$ by $\frac{b}{a}$ in \cite[Theorem 1.1]{Berndt_kim}, we can easily deduce \eqref{5.6.0}.
\end{proof}
\end{lemma}
In the following lemma, we prove the weakly increasing nature of $\{\sigma_o\text{mex}(n)\}$ and $\{\sigma_e\text{mex}(n)\}$ using a combinatorial technique. This lemma will be used to employ Theorem \ref{theorem_Ingham} in the proof of Theorem \ref{thm5.5}.
\begin{lemma}\label{lemma5.6.3}
The sequences $\{\sigma_o\emph{mex}(n)\}$ and $\{\sigma_e\emph{mex}(n)\}$ are weakly increasing.
\end{lemma}
\begin{proof}
Let $\mathcal{P}_o(n)$ (resp., $\mathcal{P}_e(n)$) denote the set of all partitions $\pi$ of $n$ with $\text{mex}(\pi)$ odd (resp., even). We construct a map
$$\Phi:\mathcal{P}_o(n)\rightarrow\mathcal{P}_o(n+1),$$
where for $\pi\in\mathcal{P}_o(n)$, if $\text{mex}(\pi)\ne1$ then $\Phi(\pi)$ is the partition of $n+1$ with $1$ added as a part to $\pi$, and if $\text{mex}(\pi)=1$ then $\Phi(\pi)$ is the partition of $n+1$ with a largest part, say $\lambda_\pi$, of $\pi$ replaced with $\lambda_\pi+1$. Notice that if $\pi_1,\pi_2\in\mathcal{P}_o(n)$ are two distinct partitions then $\Phi(\pi_1)\ne\Phi(\pi_2)$. Therefore, the map $\Phi$ is injective. Also, under the map $\Phi$, $\text{mex}(\pi)$ of any partition $\pi\in\mathcal{P}_o(n)$ remains same. Thus, for $n\geq1$
$$\sigma_o\text{mex}(n)\leq\sigma_o\text{mex}(n+1)$$
and therefore, $\{\sigma_o\text{mex}(n)\}$ is weakly increasing.
\par For $\{\sigma_e\text{mex}(n)\}$, the proof follows in the similar way with the map defined in this case as follows:
$$\Psi:\mathcal{P}_e(n)\rightarrow\mathcal{P}_e(n+1,)$$
where for $\pi\in\mathcal{P}_e(n)$, $\Psi(\pi)$ is the partition of $n+1$ with $1$ added as a part to $\pi$.
\end{proof}
We are now in a position to present a proof of Theorem \ref{thm5.5}.
%%%%%%%%%%%%%%%%%%%%%%%%%%%%%%%%%%%%%
\begin{proof}[Proof of Theorem \ref{thm5.5}]
We first recall Euler's Pentagonal Number theorem \cite[p. 12]{berndt}
\begin{align}\label{5.6.1}
(q;q)_{\infty}=\sum_{n=-\infty}^{\infty}(-1)^{n}q^{n(3n+1)/2}.
\end{align}
From \eqref{5.6.1}, we have
\begin{align}
U(q):=(q;q)_\infty&=-1+\sum_{n=0}^{\infty}(-1)^{n}q^{\frac{3}{2}n^2+\frac{1}{2}n}+\sum_{n=0}^{\infty}(-1)^{n}q^{\frac{3}{2}n^2-\frac{1}{2}n}\nonumber\\
&=-1+U_1(q)+U_2(q),\label{5.6.2}
\end{align}
where $$U_1(q)=\sum_{n=0}^{\infty}(-1)^{n}q^{\frac{3}{2}n^2+\frac{1}{2}n}\quad\text{and}\quad U_2(q)=\sum_{n=0}^{\infty}(-1)^{n}q^{\frac{3}{2}n^2-\frac{1}{2}n}.$$
Next, using Lemma \ref{lemma_Berndt_Kim} with $a=\frac{3}{2}$ and $b=\frac{1}{2}$, we obtain that as $y\rightarrow0^+$
\begin{align*}
U_1(e^{-y})&=\sum_{n=0}^{\infty}(-1)^{n}e^{-\left(\frac{3}{2}n^2+\frac{1}{2}n\right)y}\\
&=e^{y/8}\sum_{n=0}^{N}\frac{E_{2n}3^n}{(2n)!2^{3n+1}}y^n\cdot H_{2n}\left(-\sqrt{\frac{y}{6}}\right)+O(y^{N+1/2}).
\end{align*}
Using the fact that $E_0=1$ and $H_0(x)=1$, we arrive at
\begin{align}\label{5.6.3}
\lim_{y\rightarrow0^+}U_1(e^{-y})=\frac{1}{2}.
\end{align}
By repeating the similar argument we find that
\begin{align}\label{5.6.4}
\lim_{y\rightarrow0^+}U_2(e^{-y})=\frac{1}{2}.
\end{align}
Now, \eqref{5.6.2}, \eqref{5.6.3}, and \eqref{5.6.4} imply that
\begin{align*}
\lim_{y\rightarrow0^+}U(e^{-y})=0,
\end{align*}
which in turn implies that
\begin{align}\label{5.6.5}
\lim_{y\rightarrow0^+}U^2(e^{-y})=0.
\end{align}
Using modular transformation property \eqref{eta_fn_transf} of Dedekind's eta-function, it is easy to show that
\begin{align}\label{5.6.6}
\frac{1}{(e^{-y};e^{-y})_{\infty}}\sim\sqrt{\frac{y}{2\pi}}e^{\frac{\pi^2}{6y}}\quad\text{as}\ y\rightarrow0^+.
\end{align}
Using an identity:
\begin{align*}
(-q;q)_{\infty}=\frac{(q^2;q^2)_{\infty}}{(q;q)_{\infty}}
\end{align*}
and \eqref{5.6.6}, we find that
\begin{align}\label{5.6.7}
 (-e^{-y};e^{-y})^2_{\infty}=\frac{(e^{-2y};e^{-2y})^2_{\infty}}{(e^{-y};e^{-y})^2_{\infty}}\sim\left( \frac{\sqrt{\frac{y}{2\pi}}e^{\frac{\pi^2}{6y}}}{\sqrt{\frac{2y}{2\pi}}e^{\frac{\pi^2}{12y}}}\right)^2=\frac{1}{2}e^{\frac{\pi^2}{6y}}
\end{align}
as $y\rightarrow0^+$.
Finally, from \eqref{5.1.4}, \eqref{5.6.5}, and \eqref{5.6.7}, we conclude that
\begin{align}
G_o(e^{-y})&=\frac{1}{2}\left((-e^{-y};e^{-y})^2_{\infty}+(e^{-y};e^{-y})^2_{\infty}\right)\nonumber\\
&\sim\frac{1}{4}e^{\frac{\pi^2}{6y}}\label{5.6.8}
\end{align}
as $y\rightarrow0^+$.
Note that $G_o(q)$ has real nonnegative coefficients and by Lemma \ref{lemma5.6.3}, $\{\sigma_o\text{mex}(n)\}$ is weakly increasing. In view of \eqref{5.6.8}, employing Theorem \ref{theorem_Ingham} with $\mu=1/4$, $\nu=0$, and $\lambda=\pi^2/6$, we obtain that
\begin{align*}
\sigma_o\text{mex}(n)&\sim\frac{1}{8\sqrt[4]{6n^3}}\exp\left( \pi\sqrt{\frac{2n}{3}}\right)\quad\text{as}\ n\rightarrow\infty. 
\end{align*}
By the similar arguments, we arrive at
\begin{align*}
\sigma_e\text{mex}(n)&\sim\frac{1}{8\sqrt[4]{6n^3}}\exp\left( \pi\sqrt{\frac{2n}{3}}\right)\quad\text{as}\ n\rightarrow\infty. 
\end{align*}
This completes the proof of Theorem \ref{thm5.5}.
\end{proof}
%%%%%%%%%%%%%%%%%%%%

\end{document}